\documentclass[11pt,reqno]{amsart}
\usepackage[a4paper]{geometry} 
\usepackage{mathtools,amssymb,bm,amsthm} 
\usepackage{graphicx}
\usepackage{tikz-cd}
\usepackage[e]{esvect} 
\usepackage{breakcites} 
\usepackage{hyperref}
\usepackage{calrsfs}
\usepackage{soul}
\hypersetup{colorlinks=true, allcolors=blue!50!black}
\numberwithin{equation}{section}
\theoremstyle{plain}

\newtheorem{theorem}{Theorem}[section]
\newtheorem{corollary}[theorem]{Corollary}
\newtheorem{lemma}[theorem]{Lemma}
\newtheorem{proposition}[theorem]{Proposition}
\theoremstyle{definition}

\newcommand{\Sys}{\mathfrak{S}}

\newcommand{\Z}{{\mathbb Z}}

\def\<{\langle}
\def\>{\rangle}

\def\eps{\varepsilon}

\def\Z{{\mathbb Z}}

\def\FF{{\mathbb F}}

\def\Area{\mathop{\rm Area}\nolimits}

\def\1{\mathbf 1}

\makeatletter
\@namedef{subjclassname@2020}{\textup{2020} Mathematics Subject Classification}
\makeatother

\begin{document}%

\title{Short incompressible graphs and $2$-free groups}

\author[F.~Balacheff]{Florent Balacheff}

\author[W. Pitsch]{Wolfgang Pitsch}

\address{Florent Balacheff \& Wolfgang Pitsch, Departament de Matem\`atiques, Universitat Aut\`onoma de Barcelona, Barcelona, Spain}

\email{Florent.Balacheff@uab.cat \& Wolfgang.Pitsch@uab.cat}

\keywords{Incompressible graphs, $2$-free groups, systolic area, volume entropy}

\subjclass[2020]{Primary: 53C23; Secondary: 20F05, 20F34}

\thanks{The first author acknowledges support by the FSE/AEI/MICINN grant RYC-2016-19334 and by the FEDER/AEI/MICINN grant PID2021-125625NB-I00. The second author acknowledges support by the  FSE/AEI/MICINN grant PID2020-116481GB-I00. Both authors  acknowledges support by the AGAUR grant 2021-SGR-01015.}

\begin{abstract}
Consider a finite connected $2$-complex $X$  endowed with a piecewise Riemannian metric and whose fundamental group is freely indecomposable, of rank at least $3$, and in which every $2$-generated subgroup is free. In this paper we show that we can always find a connected graph $\Gamma\subset X$ such that $\pi_1 \Gamma \simeq  \FF_2\hookrightarrow\pi_1 X$ (in short, a $2$-incompressible graph) whose length satisfies the following curvature-free inequality: $\ell(\Gamma)\leq 4\sqrt{2\Area(X)}$. This generalizes a previous inequality proved by Gromov for closed Riemannian surfaces with negative Euler characteristic. As a consequence we obtain that the volume entropy of such $2$-complexes with unit area is always bounded away from zero.
\end{abstract}

\maketitle


\section{Introduction}

We are interested in the geometry of $2$-free groups. Recall that a finitely presented group $G$ is said to be \emph{$k$-free} for some $k\geq 1$ if any of its subgroups generated by $k$ elements is free (possibly of rank $\leq k$). A $1$-free group is just a group without torsion, and a $k$-free group is always $p$-free for any $p\leq k$. Obviously the free group $\FF_n$ with $n\geq 1$ generators is $k$-free for any positive $k$, and prime non-trivial examples of such groups are surface groups of genus $g\geq 2$ which are $(2g-1)$-free. 
Also, observe that the only $2$-free groups with rank at most $2$ are the free groups with one or two generators. According to \cite{AO96} the subclass of $2$-free groups is generic among groups with $3$ generators, which makes this class particularly relevant.

In order to capture this algebraic property geometrically, we first consider the various topological realizations of a group as the fundamental group of some finite $2$-complex, and then study the possible geometries that can be put on these complexes. More precisely, fix a $2$-free finitely presented group $G$ with rank at least $3$ and any finite connected $2$-complex $X$ endowed with a piecewise Riemannian metric such that $\pi_1 X=G$. 
An embedded connected graph $i:\Gamma \hookrightarrow X$ is said to be {\it $2$-incompressible} if (1) $\pi_1 \Gamma\simeq \FF_2$, and (2) the induced map $i_\ast:\pi_1 \Gamma\to \pi_1 X$
 is injective. It is worth saying that we do not require the graph to lie in the $1$-skeleton of $X$, and that we can always find $2$-incompressible graphs since loops lying in the $1$-skeleton generate the fundamental group. We then define 
\[
L_2(X):=\inf_\Gamma \ell(\Gamma)
\]
where the infimum is taken over all $2$-incompressible graphs $\Gamma$ and $\ell(\Gamma)$ denotes the total length of $\Gamma$ for the length metric induced by $X$. This is a metric invariant closely related to the {\it Margulis constant} $\mu(X)$ which is by definition the largest number $L$ such that at any point $x$ the subgroup of $\pi_1 X$ generated by loops based at $x$ with length less than $L$ is cyclic, see \cite[Definition 4.1]{Sab17}. In fact it can be easily checked that 
\begin{equation}\label{eq:L_2VsMarg}
\mu(X)\leq L_2(X)\leq 2\mu(X).
\end{equation}
The natural metric invariant $L_2$ belongs to a larger family of invariants defined as follows. For any finite connected $2$-complex $X$ endowed with a piecewise Riemannian metric define the increasing sequence of positive numbers $\{L_k(X)\}_{k\geq 1}$ by setting $L_k(X):=\inf_\Gamma \ell(\Gamma)$ where the infimum is taken over graphs which are $k$-incompressible (that is, such that $\pi_1 \Gamma\simeq \FF_k \hookrightarrow \pi_1 X$). These numbers are well defined without any particular assumption on the fundamental group of $X$  by setting $L_k(X)=\infty$ if $X$ does not admit any $k$-incompressible graph.  Observe that $L_1(X)$ is nothing but the {\it systole} of $X$ (the shortest length of a non-contractible loop) in the case where the fundamental group of $X$ is $1$-free. So the higher invariants $L_k(X)$ can be thought of as a generalization of the systole.
In this context it is natural to define for any finitely presented group $G$ its {\it $k$-free systolic area} by the formula
$$
\Sys_k(G):=\inf_{\pi_1X=G}{\Area(X)/ L_k^2(X)}
$$
where the infimum is taken over the set of finite connected $2$-complexes $X$ with given fundamental group $G$ and endowed with a piecewise Riemannian metric. 
Note that taking the supremum over the space of all piecewise flat metrics on $X$ would yield the same value, see \cite{AZ67} and \cite[\S 3]{BZ88}.
Obviously $\Sys_k(G)=0$ for any $k\geq 1$ if $G$ is free. For a $1$-free group $G$, the invariant $\Sys_1(G)$ coincides with the notion of systolic area as defined in \cite[p.337]{Gro96}. According to \cite[Theorem 6.7.A]{Gro83}, any $1$-free group $G$ which is not free satisfies the following inequality:
\[
\Sys_1(G)\geq 1/100.
\] 
The current best lower bound  known is $\pi/16$, see \cite{RS08}.
The main purpose of this article is to prove the following analog for $2$-free groups.

\begin{theorem}\label{th:L_2}
Any $2$-free group $G$ which is freely indecomposable and of rank at least $3$ satisfies the following inequality:
\[
\Sys_2(G)\geq 1/32.
\]
\end{theorem}

Therefore the new invariant $\Sys_2$ is non-trivial for a large natural class of groups. 

Theorem \ref{th:L_2} can be restated as follows: {\it any finite connected $2$-complex $X$ endowed with a piecewise Riemannian metric whose fundamental group is $2$-free and freely indecomposable, but not cyclic, satisfies the following estimate:}
\[
L_2(X)\leq 4\sqrt{2\Area(X)}.
\]
So Theorem \ref{th:L_2} generalizes the result \cite[Theorem 5.4.A]{Gro83} that any Riemannian closed orientable surface $S$ of genus at least $2$ satisfies $L_2(S)\leq 2\sqrt{2\Area(S)}$. Observe that here the assumption on the genus  ensures that the fundamental group $\pi_1 S$ is $2$-free. See also \cite[Theorem 6.6.C]{Gro83} for a higher dimensional generalization of this last inequality. Combined with inequality (\ref{eq:L_2VsMarg}), Theorem \ref{th:L_2} also provides an analog in the context of $2$-complexes of a curvature-free inequality between the volume and the Margulis constant obtained for Riemannian manifolds whose fundamental group is $2$-free, see \cite[Theorem 4.5, item (1)]{Sab17}.

Presently we do not see how to adapt our strategy to prove an analog of Theorem \ref{th:L_2} for $k>2$, but it seems reasonable to conjecture that for each such $k$ the invariant $\Sys_k$ is uniformly bounded from below for any $k$-free group freely indecomposable with rank at least $k+1$.  Also, we do not know how to extend our current proof to encompass the freely decomposable groups: a $2$-complex $X$ with decomposable fundamental group $\pi_1X= G_1 \ast G_2$ does not have to split in any meaningful way in pieces corresponding to the  subgroups $G_1$ and $G_2$.

Lastly, Theorem \ref{th:L_2} implies the following curvature-free inequality relating the volume entropy and the area. Recall that the volume entropy $h(Y)$ of a finite connected complex $Y$ (of any dimension) endowed with a piecewise Riemannian metric is defined as the exponential growth rate of the number of homotopy classes with length at most $L$, namely
\[
h(Y)=\lim_{L\to \infty} \frac{1}{L} \cdot \log \left(\text{card} \{[\gamma] \in \pi_1Y \mid \text{$\gamma$ based loop of length at most $L$}\}\right).
\]
This definition does not depend on the chosen point where loops are based.
As a consequence of Theorem \ref{th:L_2} we get the following.

\begin{corollary}\label{cor:h}
Any finite connected $2$-complex $X$ endowed with a piecewise Riemannian metric whose fundamental group is $2$-free, freely indecomposable and of rank at least $3$, satisfies the following estimate:
\[
h(X)\cdot \sqrt{\Area(X)}\geq 3\log 2/(4\sqrt{2}).
\]
\end{corollary}

There is no reason for the above constant to be optimal, but this result generalizes the following (sharp) estimate \cite{Kat82} that for $S$ an orientable closed surface whose fundamental group is $2$-free the inequality $h(S)\cdot \sqrt{\Area(S)}\geq 2\sqrt{\pi}$ is always satisfied. This corollary also improves a previous result, due to Babenko and privately communicated to the authors, proving an analog lower bound with a worst constant but valid without the freely indecomposable assumption.

\noindent \textbf{Acknowledgements.} We would like to thank I. Babenko and S. Sabourau for valuable exchanges, and the two anonymous referees for their useful comments.

\section{Topology of small balls in piecewise flat $2$-complexes}

Consider a finite connected $2$-complex $X$ endowed with a piecewise flat metric, and fix a point $x$ in $X$. 
In this section we focus on the topology of closed balls
\[
B(x,r):= \{ y \in X \ | \ d(y,x) \leq r \}
\]
and their boundary spheres 
\[
\partial B(x,r) := \{ y \in X \ | \ d(y,x) = r \}
\]
for relatively small radius $r>0$. 
	 
Our starting point is the following result proved in ~\cite[Corollary 6.8]{KRS},  for which it is crucial that the metric is piecewise flat and not just piecewise smooth.

\begin{proposition}\label{prop:nicelevelsets}
For any $r >0$, the triangulation of $X$ can be refined in such a way that both $B(x,r)$ and $\partial B(x,r)$ are CW-subcomplexes of $X$.
\end{proposition}

As a direct consequence we find that

\begin{proposition}\label{lem:topologyballgood}
For any $r >0$ and any $x \in X$, the fundamental group of $B(x,r)$ is finitely presented. 
\end{proposition}
\begin{proof}
	According to ~Proposition~\ref{prop:nicelevelsets} choose a refinement of the triangulation of $X$ such that $B(x,r)$ is a CW-subspace of $X$. Since $X$ is compact, any triangulation contains finitely many simplices, as does the triangulation of the closed ball $B(x,r)$. Hence its fundamental group  is finitely presented.  
	\end{proof}

We now turn to the boundary  spheres and show that they generically admit trivial tubular neighborhoods. 

\begin{proposition}\label{prop:boundarytunneugh}
	For all but finitely many values of $r > 0$, the boundary sphere $\partial B(x,r)$ is a finite graph, and for each connected component $C$ of $\partial B(x,r)$, there exists an open neighborhood of $C$ in $X$ homeomorphic to $C \times ]0,1[$.
\end{proposition}
\begin{proof}
Denote by $f = d(x,\cdot): X \rightarrow \mathbb{R}_+$ the function {\it distance to the point $x$}. Recall that the Reeb space $R(f)$ is the quotient of $X$ by the relation that identifies two points $y_0$ and $y_1$ if and only if $d(x,y_0)= d(x,y_1)$ and both points belong to the same connected component of the level set $f^{-1}(f(y_0))$. The space $R(f)$ admits a length structure induced from $X$.
By construction we have a canonical projection map $p: X \rightarrow R(f) $ which is $1$-Lipschitz. We argue as in \cite[Section 4]{KRS}: the function $f$ is a semi-algebraic function, then standard arguments show that $R(f)$ is a finite graph and that $R(f)$ admits a finite subdivision such that the natural map $p$ yields a trivial bundle over the interior of each edge. For all distances $r$ but the finitely many ones corresponding to the vertices of the subdivision, if $C$ is a connected component of $f^{-1}(r)$, then by triviality of the bundle the connected component of $p^{-1}(] r-\varepsilon,r+\varepsilon[)$ containing $C$ is  an open neighborhood of $C$ of the desired form provided $\varepsilon$ is small enough. More precisely, $\varepsilon$  has to be chosen at most equal to the shortest distance from $p(C)$ to one of the two ends of the edge containing it.
\end{proof}

In the last part of this section we focus on the image in $X$ of the fundamental group of small metric balls. Consider the map $i_\ast: \pi_1(B(x,r),x) \rightarrow \pi_1(X,x)$ induced by the inclusion $B(x,r)\subset X$. 

According to \cite[Proposition 3.2]{RS08} (see also \cite{KRS}), when $\pi_1X$ is $1$-free, $\text{Im} \, i_\ast$ is trivial if the radius $r$ satisfies $r<L_1(X)/2$. The last result of this section describes how $\text{Im} \, i_\ast$ remains simple under a similar assumption on the radius.

\begin{proposition}\label{prop:ballsmallradius}
Suppose that $\pi_1 X$ is a $2$-free group and fix  $r\in (0,L_2(X)/4)$. 

Then the image of the map $i_\ast: \pi_1(B(x,r),x) \rightarrow \pi_1(X,x)$ induced by the inclusion $B(x,r)\subset X$  is either trivial, or isomorphic to $\Z$.
\end{proposition}
\begin{proof}
Suppose that $\text{Im} \, i_\ast$ is not trivial. We first prove that $\text{Im} \, i_\ast$ is locally cyclic, that is, every pair of elements in the group generates a cyclic group.

For this let $\gamma_1$, $\gamma_2$  be two non-contractible loops of $X$ contained in $B(x,r)$ and based at $x$. As $\pi_1(X,x)$ is $2$-free, these loops span in $\pi_1(X,x)$ a free subgroup $H(\gamma_1,\gamma_2)$ of rank at most $2$.  Fix $\delta>0$ such that $2r+\delta<L_2(X)/2$. We first decompose each $\gamma_i$ into segments of length at most $\delta$. Then for $i=1,2$ write $\gamma_i$ as a concatenation of loops  $c_{i,1}\ast\ldots \ast c_{i,n_i}$ based at $x$ where each $c_{i,k}$ is made of the union of one of these small segments together with two shortest paths from its extremal points to $x$. Any of these loops $c_{i,k}$ based at $x$ lies by construction in $B(x,r)$ and has length at most $2r+\delta<L_2(X)/2$. So a graph made of the union of any two of these loops is of total length $< L_2(X)$, hence the subgroup in $\pi_1(X,x)$ generated by any of these pairs of loops is cyclic (if not zero). Then the subgroup $H(\{c_{i,j}\})$ generated by all the homotopy classes of the loops $\{c_{i,j}\}$ is abelian as its generators pairwise commute. In particular  there exists some positive $k$ such that $H(\{c_{i,j}\})\simeq  \Z^k$ as $\pi_1X$ is torsion-free. But $\pi_1X$ is also $2$-free so that $k=1$. It implies that  $H(\gamma_1,\gamma_2)=\Z$ and hence $\text{Im} \, i_\ast$ is locally cyclic. 

As $\text{Im} \, i_\ast$ is also finitely generated according to Proposition \ref{lem:topologyballgood}, we deduce that it is cyclic. Furthermore, as $\text{Im} \, i_\ast$ has no torsion, since $\pi_1X$ is torsion-free, it is isomorphic to $\Z$.
\end{proof}

\section{Geometry of small balls in piecewise flat $2$-complexes}

In this section we prove the central technical result of this paper. 

Consider a finite connected $2$-complex $X$ endowed with a piecewise flat metric and whose fundamental group is $2$-free, freely indecomposable and of rank at least $3$. Fix $\eps>0$ and let $\Gamma$ be a $2$-incompressible graph whose length satisfies $\ell(\Gamma)\leq L_2(X)+\eps$. Observe that $\Gamma$ may be choosen with no vertex of degree $1$. Let $x$ be any point on $\Gamma$.

\begin{theorem}\label{th:geomballs}
For all but finitely many values of $r \in (\eps,L_2(X)/4)$, the following inequality holds true:
\[
\ell (\partial B(x,r))\geq  r-\eps.
\]
\end{theorem}

In particular, using the coarea formula \cite[3.2.11]{Fed69}, we derive the following lower bound:
\[
\Area (B(x,L_2(X)/4))\geq  (L_2(X)-\eps)^2/32.
\]
This implies that 
\[
\Area(X)\geq L_2(X)^2/32,
\]
which still holds true for piecewise smooth Riemannian metrics by approximation (see \cite{AZ67} and \cite[\S 3]{BZ88}) and implies Theorem \ref{th:L_2}.
\begin{proof}
Fix $r\in (\eps,L_2(X)/4)$ so that Proposition \ref{prop:boundarytunneugh} applies and set $B:=B(x,r)$.
Denote by $X_1,\ldots,X_k$ the path connected components of $X\setminus \text{int}(B)$ with non-empty interior, and  by $C_1,\ldots,C_n$ the connected components of $\partial B$.
According to Proposition \ref{prop:nicelevelsets}, each $C_i$ is a connected finite graph, and there exists an open neighbourhood $U$ of $C_1\sqcup \ldots\sqcup C_n$ in $X$ such that
\[
U \overset{\text{hom}}{\simeq} (C_1\times ]0,1[)\sqcup \ldots\sqcup (C_n\times ]0,1[).
\]
According to Proposition \ref{prop:ballsmallradius}, the inclusion $i: B\hookrightarrow X$ induces a homomorphism of fundamental groups whose image is either trivial or isomorphic to $\Z$. So each graph $C_i$ satisfies either
$i_\ast(\pi_1 C_i)=0$ or $i_\ast(\pi_1 C_i)= \Z$. 
Furthermore, if $\text{rank} \, i_\ast(\pi_1 C_i)=\text{rank} \, i_\ast(\pi_1 C_j)=1$, then the subgroup generated by both these subgroups is a subgroup of $i_\ast(\pi(B)) = \Z$ and hence is again isomorphic to $\Z$.  In particular elements in $ i_\ast(\pi_1 C_i)$ commute with those in $i_\ast(\pi_1 C_j)$.

Let $Y=(X_1\sqcup \ldots\sqcup X_k)/\sim$ where  $x\sim y$ if and only if $x$ and $y$ belong to the same connected component $C_i$ for some $i\in \{1,\ldots,n\}$. Denote by $a_1,\ldots,a_n$ the points in $Y$ that are images of the boundary graphs $C_1,\ldots,C_n$ under the projection  map
\[
f:X_1\sqcup \ldots\sqcup X_k\to Y.
\]
The space $Y$ decomposes into a disjoint union
\[
Y_1\sqcup \ldots \sqcup Y_k
\]
of path-connected components $Y_1,\ldots,Y_k$ such that $X_j=f^{-1}(Y_j)$. 
Define for each $j=1,\ldots,k$ the subset $I_j\subset \{1,\ldots,n\}$ such that 
$a_l \in Y_j \Leftrightarrow l \in I_j$. Therefore $\{1,\ldots,n\}=I_1\sqcup\ldots\sqcup I_k$, and
\[
B \cap X_j=\sqcup_{l \in I_j} C_l.
\]

If $k=n$, we may assume, up to reindexing the boundary graphs that $a_j \in Y_j$ for each $j=1,\ldots,n$ (or equivalently that $I_j=\{j\}$). 

If $k<n$ then $|I_j|\geq 2$ for some $j\in \{1,\ldots,k\}$ and the following holds true.

\begin{lemma}\label{lem:noZ}
	Assume that $|I_j|\geq 2$. Then $i_\ast \, (\pi_1 C_l)=\Z$ for all $l \in I_j$.
\end{lemma}
\begin{proof}
	By contradiction, let $l \in I_j$ be such that $i_\ast \, (\pi_1 C_l)=0$ and fix a neighborhood $U_l$ of $C_l$ such that $U_l\simeq C_l \times]0,1[$. By construction  $U_l$ is connected, $X = U_l \cup (X \setminus C_l)$,  and because $|I_j| \geq 2$ the open set $X \setminus C_l$ is also connected.
	Observe that  $A_l:=U_l \cap (X \setminus C_l)$ has exactly two connected components and choose a point $x_1$ and $x_2$ in each one of them.  
	Fix a path $\beta$ in $U_l$ and and a path $\gamma$ in $X \setminus C_l$ both going from $x_1$ to $x_2$.
	We denote by $\varphi_1: \pi_1(A_l,x_1) \to \pi_1(U_l,x_1)$ and $\psi_1: \pi_1(A_l,x_1) \to\pi_1(X \setminus C_l, x_1)$ the homomorphisms induced by the respective inclusion maps, and we define two homomorphisms $\varphi_2: \pi_1(A_l,x_2) \to \pi_1(U_l,x_1)$ and $\psi_2: \pi_1(A_l,x_2) \to\pi_1(X \setminus C_l, x_1)$ by setting
	\[
	\varphi_2(\alpha)=\beta \alpha \beta^{-1} \, \, \text{and} \, \, \psi_2(\alpha)=\gamma \alpha  \gamma^{-1}.
	\]
We also define a homomorphism $\mu : \Z\simeq \langle a\rangle \to \pi_1(X,x_1)$ by setting
	\[
	\mu(a)=\beta\gamma^{-1}.
	\]
	By the Van Kampen theorem \cite[p.422, Proposition 2]{bourbaki}, there exists a unique surjective homomorphism
	\[
	M: \pi_1(U_l,x_1) \ast \pi_1(X\setminus C_l,x_1) \ast \Z \rightarrow \pi_1(X,x_1)
	\]  
	which coincides with $\mu$ on the factor $\Z$ and with the homomorphisms induced by the respective natural inclusions on the two factors $\pi_1(U_l,x_1)$ and $\pi_1(X\setminus C_l,x_1)$, and whose kernel is normally generated by the elements of the form
	\begin{enumerate}
		\item $\varphi_2(v) a \psi_2(v)^{-1}a^{-1}$ for $v \in \pi_1(A_l,x_2)$;
		\item $\varphi_1(v)\psi_1(v)^{-1}$ for $v \in \pi_1(A_l,x_1)$.
	\end{enumerate}
	
As the image of $\pi_1C_l\simeq \pi_1(U_l,x_1) $ is trivial in $\pi_1(X,x_1)$, the homomorphisms $M\circ \varphi_1$ and $M\circ \varphi_2$ are trivial, and consequently the surjective homomorphism $M$ factorizes as
	\[
	M: \pi_1(X\setminus C_l,x_1) \ast \Z \rightarrow \pi_1(X,x_1)
	\]   
	with kernel normally generated by the elements of the form
	\begin{enumerate}
		\item $\psi_2(v)$ for $v \in \pi_1(A_l,x_2)$;
		\item $\psi_1(v)$ for $v \in \pi_1(A_l,x_1)$.
	\end{enumerate}
By definition all these relations are written in the group $\pi_1(X\setminus C_l,x_1)$. So if we denote by $H$ the quotient of $\pi_1(X\setminus C_l,x_1)$ by these relations, $M$ induces an isomorphism
\[
\overline{M} : H \ast \Z \rightarrow \pi_1(X,x_1)
\]
contradicting the fact that the fundamental group of $X$ is freely indecomposable and of rank at least $3$.
\end{proof}
We may assume that $\Gamma$ is transverse to $C_1\sqcup\ldots \sqcup C_n$. Because it has no vertex of degree $1$, $\Gamma$ is  one of the three following graphs with first Betti number equal to $2$:   
\begin{center}
\includegraphics[scale=1.2]{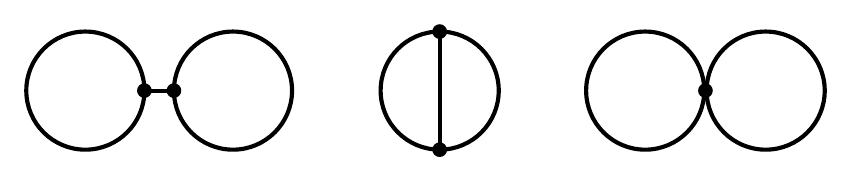}
\end{center} 
As the graph $\Gamma$ is $2$-incompressible, the subgraph $\Gamma\cap B$ has cyclic number at most $1$ according to Proposition \ref{prop:ballsmallradius}, and the graph $\Gamma$ escapes from $B$ and so necessarily  intersects the boundary $C_1\sqcup\ldots \sqcup C_n$.  Set $\Gamma_j:=\Gamma\cap X_j$ and observe that some of these graphs may be empty (but not all).
Furthermore let $\Gamma_0=\Gamma\cap B$ be the remaining part of the graph $\Gamma$ which completes the decomposition as follows: 
\[
\Gamma=\Gamma_0\cup\Gamma_1\cup \ldots \cup\Gamma_k.
\]
Now construct a new graph $\overline{\Gamma}$ starting from $\Gamma$, and obtained by deleting $\Gamma_0$ and pasting all the boundary graphs as follows:
\[
\overline{\Gamma}:=(\Gamma\setminus\Gamma_0)\cup(C_1\cup\ldots\cup C_n).
\]
We shall see that we can always extract from $\overline{\Gamma}$ a $2$-incompressible subgraph $\Gamma'$, and this implies the desired lower bound. Indeed the $2$-incompressible subgraph $\Gamma'$ will satisfy $\ell(\Gamma')\geq L_2(X)$
 as well as $\ell(\Gamma')\leq \ell(\Gamma) -r +\sum_{j=1}^n\ell(C_j)$ as $\ell(\Gamma_0)\geq r$.
Given that $\ell(\Gamma)\leq L_2(X)+\eps$, we get the announced lower bound
\[
\ell(\partial B)\geq  \sum_{j=1}^n\ell(C_j)\geq r-\eps.
\]
To  extract the $2$-incompressible subgraph $\Gamma'$ from $\overline{\Gamma}$, we argue as follows.

{\it Suppose first that the inclusion $B\subset X$ induces the zero morphism: $i_\ast(\pi_1 B)=0$.}

In particular any boundary component $C$ satisfies $i_\ast(\pi_1C)=0$ as its fundamental group  factors through $i_\ast(\pi_1B)$. Thus Lemma \ref{lem:noZ} implies that $k=n$. The key point is that there exists a unique $j\in \{1,\ldots,n\}$ such that $i_\ast(\pi_1 X_j)\neq 0$. Indeed, given that $i_\ast(\pi_1 B)=0$ and applying the Van Kampen theorem to the covering $\{B,X_1,\ldots,X_n\}$ of $X$, we get that $\pi_1X \simeq \pi_1 X_1 \ast \ldots \ast \pi_1 X_n$.   As $\pi_1 X$ is freely indecomposable, only one of these free factors is non-trivial.
  \begin{center}
      \includegraphics[scale=1.2]{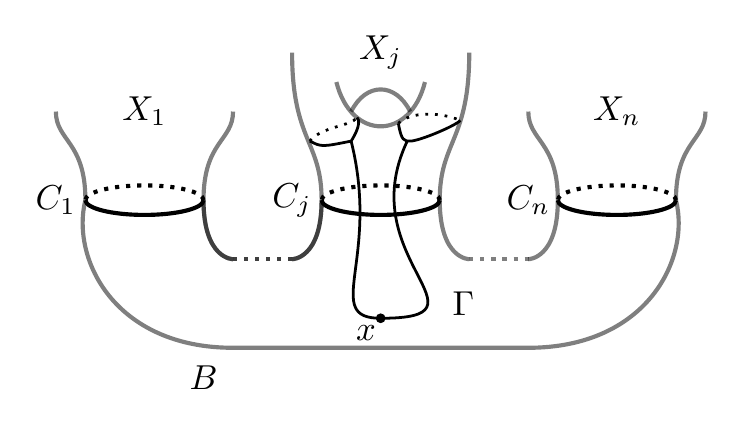}
  \end{center}
So the $2$-incompressible graph $\Gamma$, which has cyclic number $2$, must intersect the boundary graph $C_j$ of the non-trivial piece $X_j$. Fix two homotopically independent loops $c_1$ and $c_2$ of $\Gamma$ based at the same point, say $p$, of the boundary graph $C_j$. By homotopically independent we mean that the two loops generate a free subrgoup of rank $2$ of the fundamental group. If they are not entirely contained in $X_j$, and as $\pi_1(B\cup(\cup_{l\neq j} X_l))=0$, we can for each of the $c_i$'s homotope each of their subarcs  lying outside $X_j$ into a subarc of $C_j$ without moving their respective endpoints. Therefore we can homotope $c_1$ and $c_2$ into two new homotopically independent loops still  based at $p$   and lying in $\Gamma_j \cup C_j\subset \overline{\Gamma}$. Therefore, as wanted, we can extract a $2$-incompressible subgraph from $\overline{\Gamma}$.

{\it Suppose now that the inclusion $B\subset X$ induces a morphism of rank $1$: $i_\ast(\pi_1 B)=\Z$.}

Fix an element $a$ of $\pi_1B$ that generates $i_\ast(\pi_1 B)=\Z$  and a closed curve $c$ of $\Gamma$  based at $x$ and homotopically independent from $a$. 
The loop $c$ necessarily escapes from $B$. Denote by $p_1,\ldots,p_N$ the intersection points along $c$ with $\partial B$ (it may happen that $p_i=p_{i+1}$ for some $i$). Denote by $\delta_1$ the subpath of $c$ that goes from $x$ to $p_1$, by $\delta_N$ the subpath of $c$ going backwards from $x$ to $p_N$, and fix for $i=2,\ldots,N-1$ any path $\delta_i$ contained in  $B$ from $x$ to $p_i$. We can decompose the loop $c$ into a concatenation of loops $c_i$ based at $x$, each one being made by first following $\delta_i$, then the portion denoted by $\eta_i$ of  $c$ from $p_i$ to the next intersection point $p_{i+1}$, and then going back to $x$ using $\delta_{i+1}^{-1}$. One of these loops must be homotopically independent from the generator $a$ of $\pi_1B$: the loop $c$ does not homotopically commute with $a$, and thus at least one of the $c_i$'s does not homotopically commute with $a$ too. Again, this loop that we denote simply by $c_i$ necessarily  escapes from $B$ and the corresponding portion $\eta_i$ lies outside $\text{int}(B)$. Let $X_j$ be the path connected component of $X\setminus \text{int}(B)$ that contains $\eta_i$.

{\it If $X_j$ has more than one boundary component}, then by Lemma~\ref{lem:noZ} all boundary components are homotopically non-trivial in $B$ and in $X$, and we argue as follows.

Suppose first that the endpoints of  $\eta_i$ belong to two distinct boundary graphs $C_l$ and $C_{l'}$ for some $l\neq l'$. First observe that $l$ and $l'$ both necessarily belong to the same subset $I_j$ as $\eta_i \subset X_j$. Moreover $i_\ast(\pi_1 C_l)=\Z$ and $i_\ast(\pi_1 C_{l'})=\Z$ as already observed. Fix two non-trivial loops $b_l\in C_l$ and $b_{l'}\in C_{l'}$ respectively based at $p_i$ and $p_{i+1}$. Set $\delta=\delta_i^{-1}\ast\delta_{i+1}$. Observe that the homotopy classes of $\eta_i\ast b_{l'}\ast\eta_i^{-1}$ and $c_i\ast(\delta\ast b_{l'}\ast\delta^{-1})\ast c_i^{-1}$ (where $c_i$ is viewed as a loop based at $p_i$) coincide. If the loop $\eta_i\ast b_{l'}\ast\eta_i^{-1}$ was not homotopically independent with $b_l$, then we would have that $[c_i]\cdot a^n \cdot [c_i^{-1}]=a^m$ for some $m,n \in \Z\setminus \{0\}$, as both loops $\delta\ast b_{l'}\ast\delta^{-1}$ and $b_l$ induce homotopy classes in $\pi_1B=\langle a\rangle$. But this is impossible as $c_i$ was chosen homotopically independent from the class $a$. So the two loops $\eta_i\ast b_{l'}\ast\eta_i^{-1}$ and $b_l$ based at $p_i$ are homotopically independent and both contained in $\Gamma_j\cup C_l\cup C_{l'} \subset \overline{\Gamma}$. So their union forms a $2$-incompressible graph $\Gamma'\subset \overline{\Gamma}$.

Now suppose that both endpoints of  $\eta_i$ belong to the same connected boundary component $C_l$, and fix some subarc $\alpha$ in $C_l$ from $p_i$ to $p_{i+1}$. The closed curve $c_i$ (viewed as a loop based at $p_i$) is homotopic to the concatenation of the loop $\eta_i\ast\alpha^{-1}$ with the loop $\alpha\ast  \delta_{i+1}^{-1}\ast \delta_i$. The second loop is included in $B$ and therefore its homotopy class $[\alpha\ast  \delta_{i+1}^{-1}\ast \delta_i]$ is equal to $a^k$ for some $k\in \Z$. Hence  the first loop $\eta_i\ast\alpha^{-1}$ is homotopically independent from $a$. Now 
define $\Gamma''$ to be a subgraph of $C_l$ that contains $\alpha$, such that $\pi_1\Gamma''\simeq \Z$ and $\pi_1\Gamma''\to \pi_1X$ is injective. Then $\Gamma'=\Gamma''\cup \eta_i$ is the desired $2$-incompressible subgraph of $\Gamma$.


{\it If $X_j$ has a unique boundary component $C_l$}, observe that $i_\ast(\pi_1C_l)\neq 0$. For if it is trivial, by applying the Van Kampen theorem to the covering of $X$ by the open set $X\setminus X_j$ and its complement $X_j$ slightly enlarged so that these two open sets overlap along a half-tubular neighborhood $U\simeq C_l \times]0,1[$ of $C_l$, we would get a non-trivial free decomposition $\pi_1X \simeq \pi_1 X_j \ast \pi_1(X \setminus X_j)$ where both pieces are non-trivial: a contradiction. Finally, because the loop $c_i$ is homotopically independent from the class $a$, we can extract a $2$-incompressible subgraph from $C_l \cup \eta_i \subset \overline{\Gamma}$.

\end{proof}

\section{A Universal bound for the volume entropy}

We conclude by explaining how to derive Corollary \ref{cor:h} from Theorem \ref{th:L_2}.
\begin{proof}[Proof of Corollary \ref{cor:h}] 
Let $X$ be a finite connected $2$-complex $X$ endowed with a piecewise Riemannian metric whose fundamental group is $2$-free, freely indecomposable and of rank at least $3$. 
According to Theorem \ref{th:L_2} we can find a $2$-incompressible graph $\Gamma \hookrightarrow X$ with induced length at most $4\sqrt{2}\sqrt{\Area(X)}$. 
The fact that $\pi_1\Gamma\simeq \FF_2$ implies by \cite{KN07} (see also \cite{Lim08}) that 
$$
\ell(\Gamma)\cdot h(\Gamma)\geq 3 \log 2
$$
where $h(\Gamma)$ denotes the volume entropy of the finite connected $1$-dimensional complex $\Gamma$ for the piecewise Riemannian metric induced by $X$.
The injection $\pi_1\Gamma \hookrightarrow \pi_1 X$ ensures that $h(X)\geq h(\Gamma)$, from which we derive the desired lowerbound:
$$
h(X)\cdot \sqrt{\Area(X)}\geq {1\over 4\sqrt{2}} \, h(\Gamma)\cdot \ell(\Gamma)\geq {3\log 2\over 4\sqrt{2}}.
$$
\end{proof}



\end{document}